\documentclass[12pt]{amsart} %

\usepackage{float}

\usepackage{epsfig}
\usepackage{tikz}
\usepackage{array,amsmath, enumerate,  url, psfrag}
\usepackage{amssymb, amsaddr, fullpage}
\usepackage{graphicx,subfigure}
\usepackage{color}
\restylefloat{figure}

\newtheorem{theorem}{Theorem}[section]

\newtheorem{thm}{Theorem}[section]
\newtheorem{lem}[thm]{Lemma}

\newtheorem{definition}[thm]{Definition}
\newtheorem{example}[thm]{Example}

\title{Every toroidal graphs without adjacent triangles is at most 8}

\author{Fangyu Tian$^{1}$\hskip 0.2in  yuxue Yin$^{2}$}

\address{
$^{1}$\small Department of Mathematics, Central China Normal University, Wuhan, Hubei, China.\\
$^2$\small Department of EE, Tsinghua University, Beijing, China.
}

\email{yinyuxue945@mail.tsinghua.edu.cn}

\begin{document}
\maketitle
\section{Abstract}
Odd coloring is a proper coloring with an additional restriction that every non-isolated vertex has some color that appears an odd number of times in its neighborhood. The minimum number of colors $k$ that can ensure an odd coloring of a graph $G$ is denoted by $\chi_o(G)$. We say $G$ is odd $k$-colorable if $\chi_o(G)\le k$. This notion is introduced very recently by Petruševski and Škrekovski, who proved that if $G$ is planar then $ \chi_{o}(G) \leq 9 $. 
A toroidal graph is a graph that can be embedded on a torus. Note that a $K_7$ is a toroidal graph, $\chi_{o}(G)\leq7$. Tian and Yin proved that every toroidal graph is odd $9$-colorable and every toroidal graph without $3$-cycles is odd $9$-colorable.
In this paper, we proved that every toroidal graph without adjacent $3$-cycles is odd $8$-colorable.   

\section{Introduction}

An odd coloring of a graph is a proper coloring with the additional constraint that each non-isolated vertex has at least one color that appears an odd number of times in its neighborhood. We use $c_o(v)$ to denote the color of $v$. A graph G is odd $k$-colorable if it has an odd $k$-coloring. The odd chromatic number of a graph G, denoted by $\chi_o(G)$, is the minimum $k$ such that G has an odd $k$-coloring. 

Odd coloring was introduced very recently by Petru$\breve{s}$evski and $\breve{S}$krekovski~\cite{petruvsevski2021colorings}, who proved that planar graphs are odd $9$-colorable and conjecture every planar graph is $5$-colorable. Petr and Portier~\cite{petr2022odd} proved that planar graphs are odd $8$-colorable. Fabrici\cite{fabrici2022proper} proved a strengthening version about planar graphs regarding similar coloring parameters. Eun-Kyung Cho~\cite{cho2022odd} focused on a sparse graph and conjectured that, for $c\ge4$, if G is a graph with $mad(G)\le \frac{4c-4}{c+1}$, then $\chi_o(G)\le c$. 
A toroidal graph is a graph that can be embedded on a torus. Note that a $K_7$ is a toroidal graph, $\chi_{o}(G)\geq7$. Tian and Yin proved that every toroidal graph is odd $9$-colorable~\cite{fangyu22a}.
In this paper, we tight the bound one step further and proved that every toroidal graph without adjacent triangles is odd $8$-colorable, as is shown in Theorem~\ref{th1}.

\begin{theorem}\label{th1}
If $G$ is a toroidal graph without adjacent triangles, then    $\chi_o(G)\leq8$.
\end{theorem}

We use $V(G)$, $E(G)$, and $F(G)$,
respectively, to represent the vertices, edges, and faces of a graph $G$. Two cycles are
{\em  adjacent} if they share at least one common edge. A {\em  $k$-vertex} ({\em  $k^{+}$-vertex} or {\em  $k^{-}$-vertex}) is a vertex of degree $k$ (at least $k$ or at most $k$). Similarly, a {\em  $k$-face} ({\em  $k^{+}$-face} or {\em  $k^{-}$-face}) is a face of degree $k$ (at least $k$ or at most $k$). A {\em  $k$-neighbor} ({\em  $k^{+}$-neighbor} or {\em  $k^{-}$-neighbor}) of $v$ is a $k$-vertex ({\em $k^{+}$-vertex} or {\em  $k^{-}$-vertex}) adjacent to $v$.  For $f\in F(G)$, $f=[v_1v_2\ldots v_k]$ denotes that $v_1,v_2,\ldots,v_k$ are the vertices lying on the boundary of $f$ in clockwise  order.  

For a $4^+$-vertex $v$, we say $v$ is {\em convenient} if $v$ is an odd vertex or an even vertex with  a $2$-neighbor, otherwise $v$ is  {\em non-convenient}.
 A {\em  $k_i$-vertex} is a  $k$-vertex incident with $i$ $2$-vertices.

\section{Proof of Theorem~\ref{th1}}

Let $G$ be a counterexample to Theorem~\ref{th1} with the minimum number of $4^+$-vertices, and subject to that, the number of $4^+$-neighbors of $4^+$-vertices $G$ is minimized, and subject to these conditions $|E(G)|$ is minimized.

\begin{lem}\label{no-3-v}
 $G$ has no $3$-vertices.
\end{lem}
\begin{proof}
  Suppose that $G$ has a $3$-vertex $v$. We assume that $v_1,v_2$ and $v_3$ are the neighbors of $v$. Let $G'=G-v$. Since $G'$ has fewer edges than $G$, $G'$ has an odd $8$-coloring $c'$ by the minimality.   Color each vertex other than $v$ in $G$  with the same color in $G'$ and color  $v$ with $\lbrack 8 \rbrack\setminus\{c'(v_1),c'(v_2),c'(v_3),c'_o(v_1),c'_o(v_2),c'_o(v_3)\}$. Since $v$ is a $3$-vertex, $v$ must have an odd coloring.  Then the coloring of $G'$ can return back to $G$, a contradiction.
\end{proof}

\begin{lem} \label{conven}
 If $v$ is convenient, then $v$ always admits an odd coloring.
 \end{lem}
 \begin{proof}
 By the definition of convenient vertex, $v$ is either an odd vertex or an even vertex
with a $2$-neighbor. If $v$ is an odd vertex, then $v$ always admits an odd coloring. If $v$ is an even vertex with a $2$-neighbor $v'$, then $v$ always admits an odd coloring via recoloring $v'$ with
a color in $[8]\setminus  \{c_o(v), c(v), c(u), c_o(u)\}$, where $u$ is the other neighbor of $v_i$, $c_o(z)$ is the odd color of $z$ for $z\in\{u,v\}$ while $c_o(z)\neq c(v')$ .
 \end{proof}

\begin{lem}\label{tool}  The following statements hold:
\begin{enumerate}[(1)]
   \item Any two $2$-vertices are not adjacent.
    \item Any two convenient vertices are not adjacent.
    \item $k$-vertex is adjacent to at most  $2k-8$ $2$-vertices (or convenient vertices) where $4\leq k \leq 7$.
   \end{enumerate}
\end{lem}
\begin{proof}
\begin{enumerate}[(1)]
    \item Suppose otherwise that two $2$-vertices $u$ and $v$ are adjacent. Let $z'$ be the other neighbor of $z$ for $z\in\{u,v\}$. Let $G'=G-v$. Then $G'$ has an odd $8$-coloring $c'$ by the minimality of $G$.   Color each vertex other than $v$ in $G$ with the same color in $G'$ and color $v$ with $[8]\setminus\{c'(u),c'(v'),c_o'(v'),c'(u')\}$.
    Then the odd $8$-coloring of $G'$ can return back to $G$,  a contradiction.

\item Suppose otherwise that convenient vertices $u$ and $v$ are adjacent. By the definition of convenient vertices, $u$ and $v$ are $4^+$-vertices. Let $G'$ be the graph obtained from $G$ by
splitting edge $uv$ with a $2$-vertex $w$. Then $G'$ has no adjacent $3$-cycles. Since $4^+$-vertices $u$ and $v$ have fewer $4^+$-neighbors in
$G'$, there is an odd $8$-coloring $c'$ of $G'$ by the minimality of $G$. Note that $c'(u)  \neq c'(v)$
since $w$ is a $2$-vertex. Let $c(z) = c'(z)$ for $z \in V(G)$. Then $c$ is an odd coloring of $G$ since
$u$ and $v$ has an odd coloring by Lemma \ref{conven}.
   \item Suppose otherwise that a $k$-vertex $v$ is adjacent to at least $2k-7$ $2$-vertices or   convenient vertices  where $4\leq k \leq 7$. 
       Let $G'=G-v$. Then $G'$ has an odd $8$-coloring $c$ by the minimality of $G$.
 We want to color $v$ such that $v$ admits an odd $8$-coloring while each other $4^+$-vertex in $G$ has the same color as in $G'$ and some $2$-vertices can be recolored such that its neighbors admit an odd coloring.
Note that if $v$ has a $2$-neighbor $v'$, then we only need to avoid the color of the other neighbor of $v'$ in the coloring $v$. After that, recolor $v'$ with the color  different from $v$ and the other neighbor of $v'$.
If $v$ has a convenient neighbor $v'$, then we only need to avoid the color of   $v'$ in the coloring $v$ since $v'$ always has odd coloring by Lemma \ref{conven}.    If $v$ has a neighbor $v'$ which is neither $2$-neighbor nor convenient neighbor, then we need to avoid two colors ($c(v')$ and $c_o(v')$)   in the coloring $v'$.
Thus, $v$ has at most $(2k-7)+2(k-(2k-7))=7$ forbidden colors,
and there is one color left for $v$. Then the coloring of $G'$ can return back to $G$, a contradiction.
   \end{enumerate}
  \end{proof}
  \begin{lem}\label{4-f}
  (1)If $f$ is a   $3$-face, then $f$ is not incident with any $2$-vertex.

  (2) If $f$ is a $4$- or $5$-face, then $f$ is incident with at most one $2$-vertex.
  \end{lem}
\begin{proof}
(1)suppose that a $3$-face $f$ is   incident with a $2$-vertex. Let $f=[uvw]$ and $w$ is a $2$-vertex. Let $G'=G-w$. Then $G'$ has an odd $8$-coloring $c$ by the minimality of $G$. Then color $w$ with color in $[8]\setminus \{c(u),c_o(u),c(v),c_o(v)\}$. Since $uv\in E(G)$,  $c(u)\neq c(v)$. Then $w$ has an odd color, a contradiction.

(2)First suppose that $4$-face $f=[v_1v_2v_3v_4]$ is incident with two $2$-vertices. By Lemma \ref{tool}(1), we assume that $v_1$ and $v_3$ are $2$-vertices. Let $G'=G-v_1$. Then $G'$ has an odd $8$-coloring $c$.   Color each vertex other than $v_1$ in $G$  with the same color in $G'$ and color  $v_1$ with $\lbrack 8 \rbrack\setminus\{c(v_2),c(v_4),c_o(v_2),c_o(v_4)\}$. Since $v_3$ is a $2$-vertex, $c(v_2)\neq c(v_4)$. Then $v_1$ also has an odd coloring.  Then the coloring of $G'$ can return back to $G$, a contradiction.

  Next suppose  that $5$-face $f = [v_1v_2v_3v_4v_5]$ is incident with two $2$-vertices. By Lemma \ref{tool}(1), we
assume that $v_1$ and $v_3$ are $2$-vertices.  Then $v_5$ and $v_4$ are convenient,  which contradicts Lemma \ref{tool}(2).
\end{proof}

 \begin{lem}\label{3-f}
 (1)Any two $4$-vertices are not adjacent.

(2) If $f$ is a $3$-face, then  $f$ is incident with at least two non-convenient $6^+$-vertices.
  \end{lem}
\begin{proof}
  (1)Suppose otherwise that two $4$-vertices $u$ and $v$ are   adjacent. Let $u_1,u_2,u_3$ be other neighbors of $u$, $v_1,v_2,v_3$ be other neighbors of $v$. By Lemma \ref{tool}(3), $4$-vertex is not adjacent to $2$-vertex. Each of $u_1,u_2,u_3,v_1,v_2$ and $v_3$ is $4^+$-vertex by Lemma \ref{no-3-v}.   Let $G'$ be the graph obtained from $G-\{u,v\}$ by   connecting $u_1u_2,u_2u_3,u_3u_1,v_1v_2,v_2v_3,v_3v_1$ if they are not already adjacent and splitting $u_1u_2,u_2u_3,u_3u_1,v_1v_2,v_2v_3,v_3v_1$ with $2$-vertices $x_1,x_2,x_3,x_4,x_5,x_6$ respectively. Since $G'$ has fewer $4^+$-vertices than $G$, $G'$ has an odd $8$-coloring by the minimality of $G$. Then color $u$ with color in $[8]\setminus\{c(u_1),c(u_2),c(u_3),c_o(u_1),c_o(u_2),c_o(u_3)\}$, color $v$ with color in $[8]\setminus\{c(v_1),c(v_2),c(v_3),c_o(v_1),c_o(v_2),c_o(v_3),c(u)\}$.  Since each of $x_1,x_2,x_3,x_4,x_5$ and $x_6$ is a $2$-vertex, $c(u_1)\neq c(u_2) \neq c(u_3)$ and $c(v_1)\neq c(v_2) \neq c(v_3)$. Then each of $u$ and $v$ has an odd color, a contradiction.

  (2)By Lemma \ref{tool}(2), $f$ is incident with at least two non-convenient vertices. By Lemma \ref{3-f}(1), $f$ is incident with at least one non-convenient $6^+$-vertex. Thus, suppose otherwise that $3$-face $f$ is incident with one non-convenient $6^+$-vertex. Let $f=[uvw]$ and $u$ is a non-convenient $6^+$-vertex. By Lemmas \ref{4-f}(1), \ref{no-3-v} and \ref{3-f}(1), one of  $v$ and $w$ is a $4$-vertex and the other is a $5$-vertex. Then $4$-vertex is adjacent to one convenient vertex, which contradicts Lemma \ref{tool}(3).
\end{proof}

\begin{lem}\label{6-v}
Let $v$ be a non-convenient $6$-vertex, $v_1,v_2,\ldots,v_6$ be neighbors of $v$, $f_i$ be the face incident with $v_i$ and $v_{i+1}$ for $1\leq i\leq6$ and $i+1=1$ if $i=6$. If $f_1$ and $f_3$ are $3$-faces, $f_4,f_5$ and $f_6$ are $4_1$-faces, then each of $v_1,v_4,v_5$ and $v_6$ is not $5_2$- or $6_4$-vertex.
  \end{lem}
\begin{proof}
 First  suppose otherwise that $v_1$ is a $5_2$- or $6_4$-vertex by symmetry. Let $w$ be the common $2$-neighbor of $v_1$ and $v_6$, $v_1',v_1'',v_1'''$ be other neighbors of $v_1$ if $v_1$ is a $6_4$-vertex, $v_1',v_1''$ be other neighbors of $v_1$ if $v_1$ is a $5_2$-vertex where $v_2'$ is a $2$-vertex. Let $G'$ be the graph from $G$ by splitting $vv_1$ with $2$-vertex $x$. Since $4^+$-vertices $v$ and $v_1$ has fewer $4^+$-neighbor in $G'$, $G'$ has an odd $8$-coloring $c'$ by the minimality of $G$. Let $c(z)=c'(z)$ for $z\in V(G)$. Since $x$ is a $2$-vertex, $c(v)\neq c(v_1)$.  By Lemma \ref{conven}, $v_1$ always has an odd color.  If $v$ has an odd color, then $G$ has an odd $8$-coloring $c$,  a contradiction. Thus, $v$ has no odd color. Then we assume that $c(v)=1$, $c(v_1)=c(v_5)=2,c(v_2)=c(v_4)=3,c(v_3)=c(v_6)=4$. Recolor $v_1$ with color in $[8]\setminus \{c_o(v_1'),c_o(v_1''),c_o(v_1'''),c(v_6),c(v),c(v_2),2\}$ if $v_2$ is a $6_3$-vertex, with color in $[8]\setminus \{c_o(v_1'),c(v_1''),c_o(v_1''),c(v_6),c(v),c(v_2),2\}$ if $v_2$ is a $5_2$-vertex. In this case, if $v_2$ has an odd color, then $G$ is odd $8$-colorable, a contradiction. Thus,  $v_2$ has no odd color in the case of $c(v)=1$. Then recolor $v$ with color in $[8]\setminus \{1,2,3,4,c_o(v_3),c(v_1)\}$. Then $G$ is odd $8$-colorable, a contradiction.

 Next  suppose otherwise that $v_6$ is a $5_2$- or $6_4$-vertex by symmetry. Let   $v_6',v_6'',v_6'''$ be other neighbors of $v_6$ if $v_6$ is a $6_4$-vertex where $v_6',v_6''$ are $2$-vertices, $v_6',v_6''$ be other neighbors of $v_6$ if $v_6$ is a $5_2$-vertex. Let $G'$ be the graph from $G$ by splitting $vv_6$ with $2$-vertex $x$. Since $4^+$-vertices $v$ and $v_6$ has fewer $4^+$-neighbor in $G'$, $G'$ has an odd $8$-coloring $c'$ by the minimality of $G$. Let $c(z)=c'(z)$ for $z\in V(G)$. Since $x$ is a $2$-vertex, $c(v)\neq c(v_6)$.  By Lemma \ref{conven}, $v_6$ always has an odd color.  If $v$ has an odd color, then $G$ has an odd $8$-coloring $c$,  a contradiction. Thus, $v$ has no an odd color. Then we assume that $c(v)=1$, $c(v_2)=c(v_4)=2,c(v_1)=c(v_5)=3,c(v_3)=c(v_6)=4$. In this case, we remove the color of $v$ and $v_6$. Then color $v_6$ with color in $[8]\setminus \{c_o(v_6'),c_o(v_6''),c(v_6'''),c_o(v_6'''),c(v_1),c(v_5),c_o(v)\}$ if $v_6$ is a $6_3$-vertex, with color in $[8]\setminus \{c(v_6),c_o(v_6'),c(v_6''),c_o(v_6''),c(v_1),c(v_5),c_o(v)\}$ if $v_6$ is a $5_2$-vertex, color $v$ with color in $[8]\setminus \{1,2,3,4,c(v_6),c_o(v_2),c_o(v_3)\}$. Then $G$ is odd $8$-colorable, a contradiction.
\end{proof}

\begin{lem}\label{5_2}
Let $uvw$ be a $2$-path, $v$ be a $2$-vertex,  $f_1$ and $f_2$ be faces incident with $uvw$. There is no case that $u,w$ are $5_2$-vertices and $f_1$ and $f_2$ are $4_1$-faces.
\end{lem}
\begin{proof}
Suppose otherwise that $u,w$ are $5_2$-vertices and $f_1$ and $f_2$ are $4_1$-faces. Let $x$ and $y$ be other vertices incident with $f_1$ and $f_2$, respectively. Since $w$ and $u$ are convenient, $x$ and $y$ are non-convenient $6^+$-vertices by Lemma \ref{tool}(2)(3).  Let $G'$ be the graph obtained from $G$ by splitting $xw$ with  $2$-vertex $z$. Since $4^+$-vertices $x$ and $w$ has fewer $4^+$-neighbors in $G'$, $G'$ has an odd $8$-coloring $c$ by the minimality of $G$.   Since $z,v$ is a $2$-vertex, $c(x)\neq c(w)\neq c(u)$. We assume that $c(x)=1,c(w)=2,c(u)=3$.  By Lemma \ref{conven}, $w$ always has an odd color.  If $x$ has an odd color in $G$, then $G$ has an odd $8$-coloring $c$,  a contradiction. Thus, $x$ has no odd color.  Let $u',u''$ be other neighbors of $u$ and $u'$ be a $2$-vertex, $w',w''$ be other neighbors of $w$ and $w'$ be a $2$-vertex.  Then recolor $w$ with color in $[8]\setminus \{1,3,2,c(y),c(w'),c(w''),c_o(w'')\}$. Then $x$ has an odd color $c(w)$ (current color). In this case, if $y$ has an odd color, then $G$ is odd $8$-colorable, a contradiction. Thus, $y$ has no odd color. Then  recolor $u$ with color in $[8]\setminus \{1,c(w),3,c(y),c(u'),c(u''),c_o(u'')\}$. Then $x$ has odd color $c(w)$ and $y$ has odd color $c(u)$ (current color). Then $G$ is odd $8$-colorable, a contradiction.
\end{proof}
\section{proof}
 We are now ready to  complete the proof of Theorem~\ref{th1}.   Let each   $x\in V(G)\cup F(G)$ has an initial charge of $\mu(x)=d(x)-4$.   By Euler's Formula, $|V(G)|+|F(G)|-|E(G)|\geq 0$. Then $\sum_{v\in V }\mu(v)+\sum_{f\in F }\mu(f)\leq0$.

Let $\mu^*(x)$ be the charge of $x\in V(G)\cup F(G)$ after the discharge procedure. To lead to a contradiction, we shall prove that $\sum_{x\in V(G)\cup F(G)}  \mu^*(x) > 0$.  Since the total sum of charges is unchanged in the discharge procedure, this contradiction proves Theorem~\ref{th1}.

Let $v$ be a  non-convenient $6$-vertex, $v_1,v_2,\ldots,v_6$ be neighbors of $v$, $f_i$ be the face incident with $v_i$ and $v_{i+1}$ for $1\leq i\leq6$ and $i+1=1$ if $i=6$. We call $v$ is {\em special } a vertex if $f_1$ and $f_3$ are $3$-faces, $f_4,f_5$ and $f_6$ are $4_1$-faces.

  {\bf {Discharging Rules}}
  \begin{enumerate}[(R1)]
    \item Every $5^+$-vertex sends $\frac{1}{2}$ to each adjacent $2$-vertex.
    \item Every $4^+$-face sends $\frac{1}{2}$ to each incident $2$-vertex.
    \item Every non-convenient $6^+$-vertex sends $\frac{1}{2}$ to each incident $3$-face or $4_1$-face.
    \item Every  convenient $k_i$-vertex $v$  sends $\frac{1}{8} $ to  adjacent non-convenient special $6$-vertex where $k\geq5, i\leq k-1$ and $i=1$ if $k=5$, $i\leq3$ if $k=6$.
    \item Every $6^+$-face sends $\frac{1}{8}$ to each incident $4^+$-vertex.
  \end{enumerate}
  Next, we check the final charge of each vertex and each face has non-negative charge.
\begin{enumerate}[1.]
  \item Let $f$ be a $3$-face. By Lemma \ref{3-f}(2), $f$ is incident with at least two non-convenient $6^+$-vertices. Then each non-convenient $6^+$-vertex sends $\frac{1}{2}$ to $f$ by (R3). Thus, $\mu^*(f)=3-4+ \frac{1}{2}\times2 =0$.

\item Let $f=[v_1v_2v_3v_4]$ be a $4$-face. By Lemma \ref{4-f}, $f$ is incident with at most one $2$-vertex. If $f$ is not incident with any $2$-vertex, then $f$ does not send any charge by Rules. Thus, $\mu^*(f)=4-4=0$. If $f$ is incident with one $2$-vertex $v_1$, then each of $v_2$ and $v_4$ are convenient by the definition of convenient vertex. Then $v_3$ is non-convenient $6^+$-vertex by Lemma \ref{tool}. By (R2) and (R3), $f$ sends $\frac{1}{2}$ to $v_1$ and $v_3$ sends $\frac{1}{2}$ to $f$. Thus,  $\mu^*(f)=4-4-\frac{1}{2}+ \frac{1}{2}=0$.
\item Let $f$ be a $5^+$-face. By Lemma \ref{tool}(1), $f$ is incident with at most $\lfloor\frac{d(f)}{2}\rfloor$ $2$-vertices. By (R2), $f$ sends $\frac{1}{2}$ to each incident $2$-vertices. If $d(f)=5$, then $f$ is incident with at most one $2$-vertex by Lemma \ref{4-f}(2). Thus, $\mu^*(f)=d(f)-4-\frac{1}{2}>0$. If $d(f)\geq 6$, then  $f$ sends $\frac{1}{8}$ to each incident $4^+$-vertex by (R5). Thus, $\mu^*(f)=d(f)-4-\lfloor\frac{d(f)}{2}\rfloor\times\frac{1}{2}-
    \lceil \frac{d(f)}{2}\rceil\times\frac{1}{8}>0$.

  \item Let $v$ be a $2$-vertex. By Lemma \ref{tool}, $2$-vertex is adjacent to the  $5^+$-vertex. By Lemma \ref{4-f}(1), $3$-face is not incident with any $2$-vertex.     By (R1), $v$ gets $\frac{1}{2}$ from each $5^+$-neighbor and incident $4^+$-face. Thus,  $\mu^*(v)=2-4+\frac{1}{2}\times4=0$.  By Lemma \ref{no-3-v}, $G$ has no $3$-vertex.
    \item Let $v$ be a $4$-vertex. By Rules,   $v$  does not send and get any discharge procedure. Thus, $\mu^*(v)=4-4=0$.

    \item Let $v$ be a $5$-vertex. By Lemma \ref{tool}(2), $v$ is incident with at most two $2$-vertices. If $v$ is not incident with any $2$-vertex, then $v$ does not send any charge. Thus, $\mu^*(v)=5-4>0$. If $v$ is a $5_2$-vertex, then
 $v$ sends $\frac{1}{2}$ to each $2$-neighbor by (R1). Thus, $\mu^*(v)=5-4-\frac{1}{2}\times2=0$. If $v$ is a $5_1$-vertex, then
 $v$ sends $\frac{1}{2}$ to each $2$-neighbor and $\frac{1}{8}$ to each   special $6^+$-neighbor by (R1) and (R4). Thus, $\mu^*(v)=5-4-\frac{1}{2}-\frac{1}{8}\times4=0$.

 \item Let $v$ be a $6$-vertex.   By Lemma \ref{tool}(2), $v$ is incident with at most four $2$-vertices. If $v$ is a $6_4$-vertex, then
  $v$ sends $\frac{1}{2}$ to each $2$-neighbor by (R1). Thus, $\mu^*(v)=6-4-\frac{1}{2}\times4=0$. If $v$ is a $6_3$-, $6_2$- or $6_1$-vertex,   $v$ sends $\frac{1}{2}$ to each $2$-neighbor and at most $\frac{1}{8}$ to each   $6^+$-neighbor by (R1) and (R4). Thus, $\mu^*(v)=6-4-\{\frac{1}{2}\times3+\frac{1}{8}\times3,
  \frac{1}{2}\times2+\frac{1}{8}\times4, \frac{1}{2}\times1+\frac{1}{8}\times5\}>0$. Thus, $v$ is non-convenient.

    Let $v_1,v_2,\ldots,v_6$ be    neighbors of   $v$, $f_i$ be the face incident with $v_i,v$ and $v_{i+1}$  for $i=\{1,2,3,4,5\}$ and $i+1=1$ if $i=6$.
   Since $G$ has no adjacent $3$-faces, $v$ is incident with at most three $3$-faces. If $v$ is incident with three $3$-faces, then $f_1,f_3,f_5$ are $3$-face. Then $v$ is incident with at most one $4_1$-face. Otherwise $v$ is incident with  two $4_1$-faces, then $f_2$ and $f_4$ are $4_1$-faces by symmetry. Then each of $v_3$ and $v_4$ is  convenient, which contradict Lemma \ref{3-f}(2). Thus, $\mu^*(v)=6-4-\frac{1}{2}\times4=0$.

  If $v$ is incident with two $3$-faces, then $f_1,f_3$ or $f_1,f_4$ are $3$-faces. In the former case,   $v$ is incident with at most three $4_1$-faces.  Otherwise $v$ is adjacent to six convenient vertices,  which contradict Lemma \ref{tool}(2). If $v$ is incident with at most two $4_1$-faces, then $\mu^*(v)=6-4-\frac{1}{2}\times4=0$. Thus, $v$ is incident with three $4_1$-faces. If $f_2$ and two of $f_4,f_5,f_6$ are $4_1$-faces, then $v$ is adjacent to five convenient vertices, which contradict Lemma \ref{tool}(2). Then  $f_4,f_5,f_6$ are $4_1$-faces. Thus, $v$ is special and each of $v_1,v_4,v_5,v_6$ is convenient. By Lemma \ref{6-v}, each of  $v_1,v_4,v_5,v_6$  is neither $5_2$- not $6_4$-vertex. By (R4), each of  $v_1,v_4,v_5,v_6$ sends $\frac{1}{8}$ to $v$.  Thus, Thus, $\mu^*(v)=6-4-\frac{1}{2}\times5+\frac{1}{8}\times4=0$. In the latter case, $v$ is incident with at most two $4_1$-faces. Otherwise $v$ is incident with  three $4_1$-faces. Assume that  then $f_2,f_3$ and $f_5$ are $4_1$-faces by symmetry. Then each of $v_2,v_3,v_4,v_5$ and $v_6$ is convenient, which contradict Lemma \ref{3-f}(2). Thus, $\mu^*(v)=6-4-\frac{1}{2}\times4=0$.

If $v$ is incident with at least one $3$-face, then $v$ is incident with at most three $4_1$-faces. Otherwise $v$ is incident with   at least five convenient vertices, which contradict Lemma \ref{3-f}(2). Thus, $\mu^*(v)=6-4-\frac{1}{2}\times4=0$.

   \item Let $v$ be a $7$-vertex. By Lemma \ref{tool}(2), $v$ is adjacent to at most six $2$-vertices. If $v$ is a $7_6$-vertex, then $v$ is incident with five $6^+$-faces since $5^-$-face is incident with at most one $2$-vertex by Lemma \ref{4-f}. By (R5), each  $6^+$-face sends $\frac{1}{8}$ to $f$. By (R1) and (R4), $v$ sends $\frac{1}{2}$ to each $2$-neighbor and $\frac{1}{8}$ to each  special $6^+$-neighbor. Thus, $\mu^*(v)=7-4-\frac{1}{2}\times6 -\frac{1}{8}+\frac{1}{8}\times5>0$. If $v$ is a $7_i$-vertex for $i\leq5$, then $\mu^*(v)=7-4-\frac{i}{2} -\frac{1}{8}\times (7-i)>0$. If $v$ is not adjacent to any $2$-vertex, then  $v$ does not send any charge. Thus, $\mu^*(v)=7-4>0$.
   \item Let $v$ be a $8^+$-vertex.  If $v$ is  even and not adjacent to any $2$-vertex, then $v$ is non-convenient. By (R3), $v$ sends at most $\frac{1}{2}$ to each incident face. Thus,  $\mu^*(v)=d(v)-4-\frac{1}{2}d(v)\geq0$. If $v$ is odd and not adjacent to any $2$-vertex, then  $v$ does not send  any charge. Thus,  $\mu^*(v)=d(v)-4>0$. If $v$ is  adjacent to some $2$-vertices, then $v$ is  convenient. By (R1) and (R4), $v$ sends at most $\frac{1}{2}$ to each neighbor. Thus,  $\mu^*(v)=d(v)-4-\frac{1}{2}d(v)\geq0$.
\end{enumerate}

   Finally, we show that $\sum_{x\in V(G)\cup F(G)}  \mu^*(x) > 0$.

   By above checking, if $v$ is  a convenient $5^+$-vertex  and $\mu^*(v)=0$ if and only if $v$ is a $5_2$-, $5_1$-, $6_4$- or $8_8$-vertex. If $v$ is a $6_4$- or $8_8$-vertex, then $v$ is incident with a $6^+$-face. By above checking, final charge of a $6^+$-face is more than $0$. Thus, if $G$ has a $6_4$- or $8_8$-vertex, then $\sum_{x\in V(G)\cup F(G)}  \mu^*(x) > 0$. By Lemma \ref{tool}(3), any $4$-vertex is not adjacent to $2$-vertex.  Thus, every convenient vertex in $G$ is $5_2$- or $5_1$-vertex.

   If $G$ has a special $6$-vertex $v$. Let $v_1,v_2,\ldots,v_6$ be neighbors of $v$, $f_i$ be the face incident with $v_i$ and $v_{i+1}$ for $1\leq i\leq6$ and $i+1=1$ if $i=6$, $f_1$ and $f_3$ be $3$-faces, $f_4,f_5$ and $f_6$ be $4_1$-faces. Note that $v_5$ is incident with at least two $2$-vertices. By Lemma \ref{6-v}, $v_5$ is not a $5_2$-vertex. Thus, $\mu^*(v_5)>0$. Then $\sum_{x\in V(G)\cup F(G)}  \mu^*(x) > 0$. Thus, $G$ has no special $6$-vertex.  Then $5_1$-vertex only sends charge to $2$-vertex. Thus, final charge of a $5_1$-vertex is more than $0$. Thus, every convenient vertex in $G$ is $5_2$-vertex.

   By above checking,  if $f$ is a face and  $\mu^*(f)=0$ if and only if $f$ is a $3$-
    or $4$-face. If $G$ has a $2$-vertex $v$, then $v$ is incident with two $4_1$-face. Let $w,u$ be two neighbors of $v$. Then  each of $u$ and $v$ be $5_2$-vertex, and each face incident with $v$ is $4_1$-face, which contradicts Lemma \ref{5_2}. Thus, $G$ has no $2$-vertex and has no $4_1$-face.  
    Then  $\mu^*(v)\geq d(v)-6-\frac{1}{2}\times\frac{d(v)}{2}>0$ for each non-convenient $6^+$-vertex since $v$ just sends charge to each incident $3$-face. By Lemma \ref{tool}(2) and \ref{3-f}(1), $G$ has a non-convenient $6^+$-vertex. Thus, $\sum_{x\in V(G)\cup F(G)}  \mu^*(x) > 0$.

\end{document}